\documentclass[12pt, reqno]{amsart}
\usepackage[cp1251]{inputenc}
\usepackage[english]{babel}
\usepackage{amssymb, amsmath, amsfonts, graphicx}
\usepackage[colorlinks=true, citecolor=blue, urlcolor=blue]{hyperref}

\newtheorem{thrm}{Theorem}
\newtheorem{lemma}{Lemma}

\begin{document}

\title{Steffen's flexible polyhedron is embedded. \\
A proof via symbolic computations}
\author{Victor Alexandrov}
\address{Sobolev Institute of Mathematics, Novosibirsk, Russia and 
Department of Physics, Novosibirsk State University, Novosibirsk, Russia}
\email{alex@math.nsc.ru\newline 
\indent 
\href{https://orcid.org/0000-0002-6622-8214}{\rm{ORCID iD: 0000-0002-6622-8214}}}
\author{Evgenii Volokitin}
\address{Sobolev Institute of Mathematics, Novosibirsk, Russia and 
Department of Physics, Novosibirsk State University, Novosibirsk, Russia}
\email{volok@math.nsc.ru\newline 
\indent 
\href{https://orcid.org/0000-0002-2646-7800}{\rm{ORCID iD: 0000-0002-2646-7800}}}
\date{August 4, 2025}

\begin{abstract} 
A polyhedron is flexible if it can be continuously deformed preserving the shape 
and  dimensions of every its face. 
In the late 1970's Klaus Steffen constructed a sphere-homeomorphic embedded 
flexible polyhedron with triangular faces and with 9 vertices only,
which is well-known in the theory of flexible polyhedra.
At about the same time, a hypothesis was formulated that the Steffen polyhedron 
has the least possible number of vertices among all embedded flexible 
polyhedra without boundary. 
A counterexample to this hypothesis was constructed by Matteo Gallet, Georg Grasegger,
Jan Legersk{\'y}, and Josef Schicho in 2024 only.
Surprisingly, until now, no proof has been published in the mathematical literature 
that the Steffen polyhedron is embedded.
Probably, this fact was considered obvious to everyone who made a cardboard model 
of this polyhedron.
In this article, we prove this fact using computer symbolic calculations.
\end{abstract}
\maketitle

\section{Introduction}\label{sec1}
Let 
$K$ 
be a connected 2-dimensional simplicial complex with or without boundary.
Depending on the context, a \textit{polyhedron} is either a continuous map 
$f: K\to\mathbb{R}^3$ 
which is affine linear and nondegenerate on every simplex or the image  
$f(K)\subset\mathbb{R}^3$ 
of 
$K$. 
A polyhedron is \textit{embedded} (or \textit{self-intersection free}) if 
$f$ 
is injective.
A polyhedron 
$P_0:K\to\mathbb{R}^3$ 
is \textit{flexible} if there are
$\varepsilon>0$ 
and a continuous family 
$\{P_t \}_{t\in(-\varepsilon,\varepsilon)}$ 
of polyhedra 
$P_t: K\to\mathbb{R}^3$ 
such that,
for every 
$t\neq 0$, $P_t(\sigma)$ 
is congruent to  
$P_0(\sigma)$ 
for every  
$\sigma\in K$, 
while 
$P_t(K)$ 
and 
$P_0(K)$ 
themselves are not congruent to each other.
The family 
$\{P_t\}_{t\in (-\varepsilon, \varepsilon)}$ 
is a (nontrivial) \textit{flex} of
$P_0$, 
and 
$t$ 
is a \textit{parameter} of the flex.

The first examples of flexible polyhedra without boundary in 
$\mathbb{R}^3$ 
were constructed by Raoul Bricard in 1897 in \cite{Br97}. 
Nowadays they are called \textit{Bricard octahedra} since, for all of them,
$K$ 
is equivalent to the natural simplicial complex of a regular octahedron. 
Every Bricard octahedron has self-intersections.

The first example of an embedded flexible polyhedron without boundary in 
$\mathbb{R}^3$ 
was constructed by Robert Connelly in 1977 in \cite{Co77}; 
it has 18 vertices and its simplicial complex 
$K$ 
is homeomorphic to the sphere.

In 1978, Klaus Steffen constructed an example of an embedded sphere-homemorphic
flexible polyhedron with only 9 vertices.
Though Steffen never published his example in a mathematical journal, nowadays
it is widely known as the Steffen polyhedron.
For a long period of time the Steffen polyhedron was supposed to have 
the least possible number of vertices among all embedded flexible polyhedra
without boundary.
A conterexample was constructed by Matteo Gallet, Georg Grasegger, Jan Legersk{\'y}, 
and Josef Schicho in 2024 only, see \cite{GGLS24}.

Surprisingly, until now, no proof has been published in the mathematical literature 
that the Steffen polyhedron is embedded.
Probably, the reason is that this fact is considered obvious to everyone who made a
cardboard model of this polyhedron. 

In this article, we give the first proof of this fact.
Our proof is based on computer symbolic calculations.
The corresponding algorithm was previously presented in our article 
\cite{AV24} and post \cite{AV25}.

\section{The Steffen polyhedron $S_0$}\label{sec2} 

In Section \ref{sec2} we briefly explain what the Steffen polyhedron is. 
The easiest way to do this is to explain how to build its cardboard model.
Another approach to introduce the Steffen polyhedron to the reader is realized 
in \cite{Al10}.

A cardboard model of the Steffen polyhedron 
can be glued from the development shown in Fig.~\ref{fig1}. 
The gluing instructions and explanations are given in the capture under the figure.

\begin{figure}[h]
\begin{center}
\includegraphics[width=0.8\textwidth]{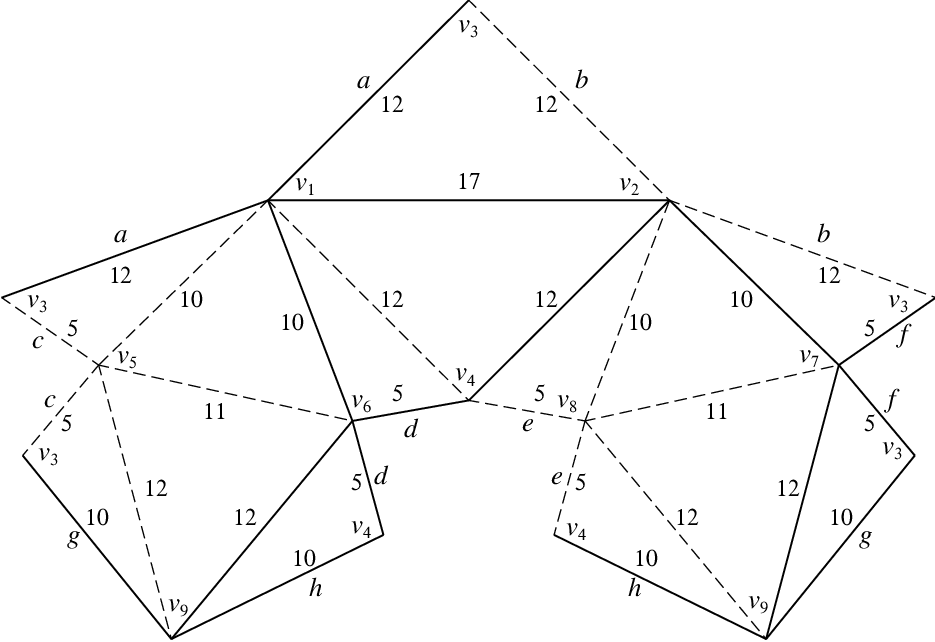}
\end{center}
\caption{Development of the Steffen polyhedron.
Solid lines represent mountain folds, dash lines are valley folds.
Integer numbers (other than subscripts) indicate edge lengths.
Symbols 
$v_j$ 
indicate vertices, 
$j=1,\dots,9$.
Letters 
$a,\dots,h$, 
written outside the development, provide gluing instructions.}\label{fig1}
\end{figure}

When gluing, it is useful to keep in mind the following two facts, 
for which we refer the reader to \cite{Al10}:

(a) after performing all the gluings 
$a,\dots,h$, 
the spatial distance between the vertices 
$v_3$ 
and 
$v_4$ 
is automatically set to 11 (this is not immediately obvious from Fig.~\ref{fig1}, because 
$v_3$ 
and 
$v_4$ are not connected by an edge);

(b) after performing the gluings 
$a,\dots,f$, 
but not the gluings 
$g$ 
and 
$h$, 
both the vertex 
$v_9$ 
of the triangle 
$\{v_5, v_6, v_9\}$ 
and  the vertex 
$v_9$ 
of the triangle 
$\{v_7, v_8, v_9\}$ 
can move independently of each other along a circle, 
$\gamma$,
which is the intersection of two spheres of the same radius 10;
one sphere is centered at 
$v_3$, 
the other at 
$v_4$.

We call the \textit{Steffen polyhedron} 
$S_0$ 
a polyhedron glued from the development shown in Fig.~\ref{fig1} in accordance with the 
gluing instructions given in the capture under Fig.~\ref{fig1} and such that its vertex 
$v_9$ 
is located on the ray coming out of the midpoint of the segment 
$\{v_3, v_4\}$ 
and passing through the the midpoint of the segment 
$\{v_1, v_2\}$. 

Let us introduce a right-handed Cartesian coordinate system in Euclidean 
3-space such that its origin $0$ coincides with the midpoint of 
$\{v_3, v_4\}$
and its positive 
$x$- 
and $z$-semiaxis pass through the point 
$v_4$ 
and the middle point of 
$\{v_1, v_2\}$, 
respectively; see Fig.~\ref{fig2}.
In this article, we use coordinates of points relative to this coordinate system only.

\begin{figure}[h]
\begin{center}
\includegraphics[width=0.3\textwidth]{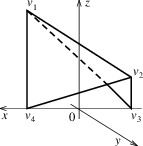}
\end{center}
\caption{The right-handed Cartesian coordinate system in Euclidean 3-space
associated with the Steffen polyhedron 
$S_0$.}\label{fig2}
\end{figure}

By definition, put 
$v_j=(x_j, y_j, z_j)$ 
for every 
$j=1,\dots, 9$.

The coordinates of the vertices 
$v_1, \dots, v_4$, and $v_9$
are known to us from the above:
\begin{equation}\label{eq1}
v_1 = (0, -\frac{17}{2}, \frac{\sqrt{166}}{2}); \ 
v_2 = (0,  \frac{17}{2}, \frac{\sqrt{166}}{2}); \ 
-v_3 = v_4 = (\frac{11}{2}, 0, 0); \ 
v_9 = (0, 0, -\frac{3\sqrt{31}}{2}).
\end{equation}
To find the coordinates of the remaining vertices 
$v_5, \dots, v_8$, 
we note that, for every 
$j=5,\dots, 8$, 
$v_j$ 
is connected by an edge to some three vertices among those, whose coordinates are 
specified in (\ref{eq1}).
For example, for  
$v_6$, 
these are
$v_1$, $v_4$, 
and 
$v_9$.
Therefore, to find the coordinates of 
$v_6$, 
we solve the following system of algebraic equations
\begin{gather}
(x_6-x_1)^2+(y_6-y_1)^2+(z_6-z_1)^2=10^2,\nonumber\\
(x_6-x_4)^2+(y_6-y_4)^2+(z_6-z_4)^2=5^2, \label{eq2}\\
(x_6-x_9)^2+(y_6-y_9)^2+(z_6-z_9)^2=12^2\nonumber
\end{gather}
with respect to 
$x_6$, $y_6$, 
and 
$z_6$.
We solve (\ref{eq2}) symbolically using the Wolfram Mathematica
software system and obtain two solutions 
$(x'_6, y'_6, z'_6)$
and
$(x''_6, y''_6, z''_6)$,
which correspond to points 
$v'_6$ 
and 
$v''_6$ 
in Euclidean 3-space, which obviously are symmetric to each other with respect 
to the plane, passing through the vertices 
$v_1$, $v_4$, 
and 
$v_9$.
In Fig.~\ref{fig1}, the edge 
$\{v_1,v_2\}$
is a mountain fold, while 
$\{v_1,v_4\}$
is a valley fold.
Hence, the vertices
$v_3$ 
and 
$v_6$ 
are located in the different halfspaces determined by the plane passing 
through the vertices
$v_1$, $v_2$, 
and 
$v_4$; 
equivalently, we can say that the vectors
$v_3$ 
and 
$v_6$ 
are such that the following mixed products
\begin{equation}\label{eq3}
((v_2-v_1)\times (v_4-v_1))\cdot (v_3-v_1)
\quad\mbox{and}\quad
((v_2-v_1)\times (v_4-v_1))\cdot (v_6-v_1)
\end{equation}
have opposite signs.
So, using (\ref{eq3}) and executing symbolic calculations 
with Wolfram Mathematica, we detect the vertex
$v_6=(x_6,y_6,z_6)$
among the two  solutions
$v'_6$ 
and 
$v''_6$
to (\ref{eq2}).
It turns out that
\begin{equation}\label{eq4}
\begin{split}
x_6&=\frac{258783870279 -389769468\sqrt{5146}+102\sqrt{31a}}{51998549858},\\  
y_6&=\frac{-89746193059+533205663\sqrt{5146}-33\sqrt{31a} -11\sqrt{166a}}{25999274929},\\
z_6&=\frac{6920764197\sqrt{31}+714577358\sqrt{166} -187\sqrt{a}}{25999274929}.
\end{split}
\end{equation}
where
$a= 23829556819105727+373057935372156\sqrt{5146}.$

Using similar arguments, we find expressions in radicals for the coordinates of 
the vertices 
$v_5$, $v_7$, 
and 
$v_8$
through symbolic calculations and use them in Section \ref{sec4}.
Those expressions are even longer then the expressions for the coordinates of 
$v_6$ 
in (\ref{eq4}) and, thus, are even less informative.
So, we omit them, but write the approximate values for the coordinates of 
$v_5, \dots, v_8$: 
\begin{equation}\label{eq5}
\begin{split}
v_5 \approx (-1.98248, 0.834943, 3.45397),\quad
v_6 \approx (6.89559, -4.79631, 0.218428), \\
v_7 \approx (-6.89559, 4.79631, 0.218428),\quad  
v_8 \approx (1.98248, -0.834943, 3.45397). 
\end{split}
\end{equation}

The main goal of this article is to prove that the Steffen polyhedron 
$S_0$
with the vertices 
$v_1, \dots, v_9$, 
the coordinates of which are given in radicals by (\ref{eq1}), (\ref{eq4}),
and by the expressions in radicals, corresponding to (\ref{eq5}), is embedded.
This goal will be achieved using symbolic calculations described in Sections
\ref{sec3} and \ref{sec4}.

We have already mentioned that the Steffen polyhedron 
$S_0$
is flexible.
Let us define its flex 
$\{S_t \}_{t\in(-\varepsilon,\varepsilon)}$ 
right now.
To do this, it suffice to specify the position of each vertex
$v_j(t)$
of the polyhedron
$S_t$.
By definition, we put
$v_j(t)=v_j$
for all
$j=1,\dots,4$
and all $t$ close enough to zero, and define 
$v_9(t)$ 
as the point lying on the circle
$\gamma$,
defined in the property (b) in Section \ref{sec2}, and such that the oriented angle 
$\angle v_9(t)0v_9$
is equal to
$t$ 
radians; finally, we calculate the coordinates of  
$v_j(t)$ 
for all
$j=5,\dots,8$ 
and all 
$t$ 
close enough to zero, from the coordinates of  
$v_j(t)$, $j=1,\dots,4$, 
and
$v_9(t)$ 
in the same way as the coordinates of  
$v_j$ for $j=5,\dots,8$
were calculated from the coordinates of  
$v_j$, $j=1,\dots,4$, 
and
$v_9$ 
above.
In Section \ref{sec5} we study some properties of 
$\{S_t \}_{t\in(-\varepsilon,\varepsilon)}$
using floating point calculations. 

\section{An algorithm for checking whether a polyhedron is embedded}\label{sec3} 

Many algorithms for recognizing self-intersections of polyhedral surfaces 
are described in the literature; for example, see references in \cite{AV24}.
But their creators optimize performance, memory usage, and other parameters 
that are not interesting to us; on the other hand  they often don't care about 
the missing a ``small'', indistinguishable on the screen, self-intersection.
This does not suit us in principle, because even the presence of an intersection
consisting of a single point changes the answer in our problem to the opposite.
Therefore, we had to propose our own algorithm and had to realize it in Wolfram
Mathematica, using symbolic computations only.
Lemmas~\ref{lemma1} and~\ref{lemma2} provide mathematical background for our algorithm.

\begin{lemma}\label{lemma1}
Let 
$K$ 
be a connected 2-dimensional simplicial complex with or without boundary, and let 
$f: K\to\mathbb{R}^3$ 
be a polyhedron.
Then the following statements are equivalent to each other\,$:$ 
\par
$(\mbox{\emph{i}})$ $f(K)$
is not embedded in 
$\mathbb{R}^3;$
\par
$(\mbox{\emph{ii}})$ there are two simplices 
$\sigma_1, \sigma_2\in K$, $\dim\sigma_1\leqslant 1$,
and two points 
$u_j\in\sigma_j$, $j=1,2$,
such that 
$u_1\neq u_2$, $f(u_1)=f(u_2)$.
\end{lemma}

\begin{proof}
The statement (ii) yields that
$f$ 
is not injective. 
Hence, 
$f(K)$ 
is not embedded, and (i) is true.

Conversely, suppose that (i) holds true.
Then
$f$ 
is not injective, i.e., there are two points 
$\widetilde{u}_j\in K$, $j=1,2$,
such that 
$\widetilde{u}_1\neq \widetilde{u}_2$ 
and 
$f(\widetilde{u}_1)=f(\widetilde{u}_2)$.
For every 
$j=1,2$, 
among all simplices in 
$K$
containing  
$\widetilde{u}_j$,
choose the simplex of the minimum dimension and denote it by 
$\widetilde{\sigma}_j$.
Let us show how to choose
$\sigma_1$, $\sigma_2$, $u_1$, and $u_2$, 
satisfying (ii).

First, consider the case, when 
$\dim \widetilde{\sigma}_1<2$
or 
$\dim \widetilde{\sigma}_2<2$.
Swapping, if necessary, the indices 
$j=1,2$ 
so that 
$\dim\widetilde{\sigma}_1\leqslant 1$, 
we conclude that (ii) is true with
$u_j=\widetilde{u}_j$
and
$\sigma_j=\widetilde{\sigma}_j$
for 
$j=1,2$.

Now consider the case, when 
$\dim\widetilde{\sigma}_1=\dim\widetilde{\sigma}_2=2$.
For every 
$j=1,2$, 
$\widetilde{u}_j$
is an interior point of
$\widetilde{\sigma}_j$ 
because
$\widetilde{\sigma_j}$
has the minimum dimension among all simplices in 
$K$
containing  
$\widetilde{u}_j$.
Therefore, the point 
$f(\widetilde{u}_1)=f(\widetilde{u}_2)$ 
is an interior point of both the triangle 
$f(\widetilde{\sigma}_1)$ 
and the triangle 
$f(\widetilde{\sigma}_2)$.
Hence, 
$f(\widetilde{\sigma}_1)\cap f(\widetilde{\sigma}_2)$
is not reduced to the single point
$f(\widetilde{u}_1)=f(\widetilde{u}_2)$.
Let a straight-line segment 
$\tau$
be such that 

(A) 
$\tau\subset f(\widetilde{\sigma}_1)\cap f(\widetilde{\sigma}_2)$;

(B) 
$f(\widetilde{u}_1)=f(\widetilde{u}_2)\in \tau$;  
and 

(C) 
$\tau$ 
is maximal with respect to inclusion among all segments satisfying (A) and (B).

Note that at least one endpoint of  
$\tau$ 
is not contained in the set
$f(\widetilde{\sigma}_{12})$,
where 
$\widetilde{\sigma}_{12}=\widetilde{\sigma_1}\cap\widetilde{\sigma_2}$.
Indeed, by the definition of simplicial complex, 
$\widetilde{\sigma}_{12}$ 
is either an empty set, a 0-dimensional simplex, or a 1-dimensional simplex of 
$K$, 
and, 
by the definition of polyhedron, 
$f\vert_{\widetilde{\sigma}_{12}}$
is a nondegenerate affine linear map.
Therefore, assuming that both ends of  
$\tau$ 
are contained in  
$f(\widetilde{\sigma}_{12})$, 
it follows that  
$\tau\subset f(\widetilde{\sigma}_{12})$.
But then the condition 
$f(\widetilde{u}_1)=f(\widetilde{u}_2)\in \tau$ 
implies 
$\widetilde{u}_1=\widetilde{u}_2$, 
which contradicts the above assumption 
$\widetilde{u}_1\neq\widetilde{u}_2$.

So, at least one endpoint of  
$\tau$ 
is not contained in 
$f(\widetilde{\sigma}_{12})$.
Let us denote that endpoint by
$v$.
Since
$\tau$
is maximal with respect to inclusion among all segments with the properties (A) and (B),
$v$ 
cannot be an interior point for both the triangle 
$f(\widetilde{\sigma}_1)$ 
and the triangle 
$f(\widetilde{\sigma}_2)$.
Swapping, if necessary, the indices 
$j=1,2$, 
we may assume without loss of generality that 
$v$
is not an interior point of the triangle 
$f(\widetilde{\sigma}_1)$.
Hence, there is a simplex 
${\sigma_1}\subset\widetilde{\sigma}_1$,
$\dim {\sigma}_1 <\dim \widetilde{\sigma}_1=2$,
such that
$v\in f(\sigma_1)$.
Thus, (ii) is true with
$u_1=(f\vert_{\sigma_1})^{-1}(v)\in\sigma_1$,
$u_2=\widetilde{u}_2\in\sigma_2$,
and 
$\sigma_2=\widetilde{\sigma}_2$.
\end{proof}

Informally speaking, Lemma \ref{lemma1} reduces the problem ``whether a given 
polyhedral surface has self-intersections'' to the problem ``whether a segment 
and a triangle intersect''.
Now we begin to study the last one, and we need the following notation.

As is known, if 
$x_i=(x_{i,1},x_{i,2}, x_{i,3})$, $i=0,\dots,3$, 
are four points in 
$\mathbb{R}^3$, 
then the oriented volume 
$\operatorname{Vol}(x_0, x_1, x_2, x_3)$ 
of the tetrahedron 
$\{x_0, x_1, x_2, x_3\}$
can be calculated using one of the following formulas
\begin{equation*}
\begin{split}
\operatorname{Vol}(x_0, x_1, x_2, x_3)= 
{\frac16}((x_1-x_0)\times (x_2-x_0))\cdot (x_3-x_0)\hphantom{AAAAAAAAAAAAA}\\
\hphantom{AAAAAAAA} ={\frac16}{\begin{vmatrix}
x_{1,1}-x_{0,1}&x_{1,2}-x_{0,2}&x_{1,3}-x_{0,3}\\
x_{2,1}-x_{0,1}&x_{2,2}-x_{0,2}&x_{2,3}-x_{0,3}\\
x_{3,1}-x_{0,1}&x_{3,2}-x_{0,2}&x_{3,3}-x_{0,3}
\end{vmatrix}}
={\frac16}{\begin{vmatrix}
1&x_{0,1}&x_{0,2}&x_{0,3}\\
1&x_{1,1}&x_{1,2}&x_{1,3}\\
1&x_{2,1}&x_{2,2}&x_{2,3}\\
1&x_{3,1}&x_{3,2}&x_{3,3}
\end{vmatrix}},
\end{split}
\end{equation*}
and, in particular, is a polynomial in variables 
$x_{i,k}$, $i=0,\dots,3$, $k=1,\dots,3$.
We put by definition 
$g(x_0, x_1, x_2, x_3)= 6\operatorname{Vol}(x_0, x_1, x_2, x_3)$.
 
Let
$f: K\to\mathbb{R}^3$
be a polyhedron, and let 
$\delta$
and 
$\Delta$
be 1- and 2-dimensional simplices of
$K$, 
respectively.
We denote the vertices of the closed triangle
$f(\Delta)$
by 
$y_k=(y_{k,1},y_{k,2},y_{k,3})\in\mathbb{R}^3$, $k=1,\dots,3$,
and denote the endpoints of the closed straight-line segment
$f(\delta)$
by
$z_j=(z_{j,1},z_{j,2},z_{j,3})\in\mathbb{R}^3$, $j=1,2$.

\begin{lemma}\label{lemma2}
With the above notation, the following statements are true
\par
$(\alpha)$ 
if 
$g(y_1, y_2, y_3, z_1)g(y_1, y_2, y_3, z_2)>0$, then 
$f(\delta)\cap f(\Delta)=\varnothing;$

$(\beta)$
if 
$g(y_1, y_2, y_3, z_1)g(y_1, y_2, y_3, z_2)<0$, 
and, in addition,
$g(z_1, z_2, y_2, y_3)$, $g(z_1, z_2, y_3, y_1)$,
and
$g(z_1, z_2, y_1, y_2)$
are nonzero and have one and the same sign, then 
$f(\delta)\cap f(\Delta)\neq\varnothing;$

$(\gamma)$ 
if 
$g(y_1, y_2, y_3, z_1)g(y_1, y_2, y_3, z_2)<0$, 
and, in addition,
$g(z_1, z_2, y_2, y_3)$, $g(z_1, z_2, y_3, y_1)$,
and
$g(z_1, z_2, y_1, y_2)$
are nonzero, but not all of them have one and the same sign, then 
$f(\delta)\cap f(\Delta)=\varnothing$.
\end{lemma}

\begin{proof}
The statement $(\alpha)$ is true because if 
$g(y_1, y_2, y_3, z_1)g(y_1, y_2, y_3, z_2)>0$,
then 
$z_1$ 
and 
$z_2$ 
lie on one side of the plane containing the triangle 
$f(\Delta)=\{y_1, y_2, y_3\}$
(see the left part of Fig.~\ref{fig3}).

\begin{figure}[ht]
\begin{center}
\includegraphics[width=0.9\textwidth]{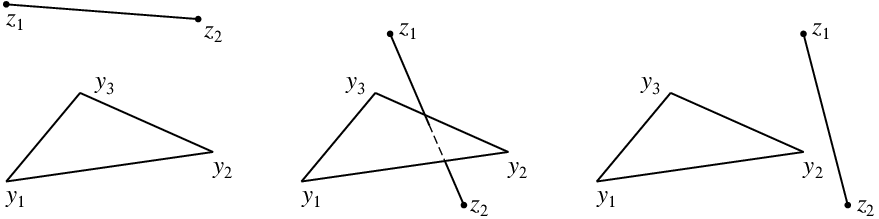}
\end{center}
\caption{Various cases of mutual arrangement of the triangle 
$f(\Delta)$ with vertices $y_1, y_2, y_3$
and the straight-line segment $f(\delta)$ with endpoints $z_1$, $z_2$.}\label{fig3}
\end{figure}

Now let us assume that the conditions of the statement 
$(\beta)$ 
are fulfilled.
Since
\begin{equation*}
g(y_1, y_2, y_3, z_1)g(y_1, y_2, y_3, z_2)<0,
\end{equation*}
the points
$z_1$ 
and 
$z_2$ 
lie on different sides of the plane containing the triangle 
$f(\Delta)=\{y_1, y_2, y_3\}$ 
(see the center and right parts of Fig.~\ref{fig3}).
For each point 
$x$ 
on the boundary of  
$f(\Delta)$, 
denote by 
$\pi(x)$ 
the halfplane bounded by the line containing 
$f(\delta)$ 
and passing through 
$x$.
Since  
$g(z_1, z_2, y_2, y_3)$, $g(z_1, z_2, y_3, y_1)$,
and
$g(z_1, z_2, y_1, y_2)$ 
have the same sign, then if 
$x$ 
goes around the boundary of 
$f(\Delta)$ 
once, moving all time in the same direction, then the halfplane 
$\pi(x)$  
also always rotates in the same direction and makes a complete turn 
around the straight line including 
$f(\delta)$. 
This means that the boundary of 
$f(\Delta)$ 
is linked to the straight line including 
$f(\delta)$.  
Taking into account that the endpoints of 
$f(\delta)$ 
lie on the opposite sides of the plane containing 
$f(\Delta)$, we conclude that 
$f(\delta)\cap f(\Delta)\neq\varnothing$.
Hence, 
$(\beta)$ is proved.
This case is schematically shown in the center part of Fig.~\ref{fig3}.

We treat the case 
$(\gamma)$ 
similarly to  
$(\beta)$.
But this time the halfplane 
$\pi(x)$  
does not always rotate in the same direction and therefore does not make 
a complete turn around the straight line including 
$f(\delta)$. 
Therefore,
$f(\delta)\cap f(\Delta)=\varnothing$,
and
$(\gamma)$ 
is proved.
This case is schematically shown in the right part of Fig.~\ref{fig3}.
\end{proof}

Lemma \ref{lemma2} solves the problem ``whether a given segment 
and triangle intersect'' in many cases, but in all.
We say that the latter are ``cases requiring additional study''.
We hope that in reality there will be few or no such cases.
Therefore, we do not want to complicate our algorithm; it suit us if, when 
such a case occurs, the algorithm informs us about it and continues its work.

Now let us describe our algorithm.

The algorithm accepts four files as input: 
a list \texttt{s} of the 1-dimensional simplices of 
$K$; 
a list \texttt{ss} of the edges of 
$f(K)\subset\mathbb{R}^3$
with the coordinates of their endpoints; 
a list \texttt{t} of the 2-dimensional simplices of 
$K$; 
and a list \texttt{tt} of the faces of 
$f(K)\subset\mathbb{R}^3$
with the coordinates of their vertices. 

Our algorithm performs a complete search for all pairs 
$(\delta, \Delta)$, 
where 
$\delta\in\tt{s}$
and 
$\Delta\in\tt{t}$.
When 
$(\delta, \Delta)$
is fixed, proceed as follows: 

$\langle 1\rangle$ 
If 
$\delta\subset \Delta$, 
then go to 
$\langle 8\rangle$; 
else go to 
$\langle 2\rangle$.
[At this step we conclude that this situation makes no contribution to the set of 
self-intersections of
$f(K)$]

$\langle 2\rangle$ 
Pick up the coordinates of the endpoints 
$z_1$
and 
$z_2$ 
of 
$f(\delta)$
from the list 
$\tt{ss}$
and the coordinates of the vertices 
$y_1$, $y_2$,
and 
$y_3$ 
of 
$f(\Delta)$
from the list 
$\tt{tt}$.
If 
$\delta\cap \Delta$
consists of a single point, 
$u$, 
then go to 
$\langle 3\rangle$; 
else go to 
$\langle 4\rangle$.
[At this step we introduce notation and switch between two cases]

$\langle 3\rangle$ 
Change, if necessary, the indices 1, 2, and 3 so that 
$y_1=z_1=f(u)$;
if 
$g(y_1,y_2,y_3,z_2)\neq 0$ 
then go to 
$\langle 8\rangle$;
else add the line ``The case of $(\delta, \Delta)$ requires additional study''
to the output file 
$\texttt{out1}$
and go to 
$\langle 8\rangle$.
[At this step we study the case when 
$\delta\cap \Delta$
consists of a single point]

$\langle 4\rangle$
If
$g(y_1, y_2, y_3, z_1)g(y_1, y_2, y_3, z_2)>0$,
then go to 
$\langle 8\rangle$; 
else go to 
$\langle 5\rangle$.
[At this step we study the case when
$\delta\cap \Delta=\varnothing$; 
according to the case 
$(\alpha)$ 
of Lemma \ref{lemma2},
we conclude that that this situation makes no contribution to the set of 
self-intersections of
$f(K)$]

$\langle 5\rangle$
If
$g(y_1, y_2, y_3, z_1)g(y_1, y_2, y_3, z_2)<0$ 
and
$g(z_1, z_2, y_2, y_3)$, $g(z_1, z_2, y_3, y_1)$,
and
$g(z_1, z_2, y_1, y_2)$
are nonzero and have one and the same sign, then 
add the line ``The edge 
$f(\delta)$ 
intersects the face
$f(\Delta)$''
to the output file 
$\texttt{out2}$,
and go to 
$\langle 8\rangle$;
else go to 
$\langle 6\rangle$.
[At this step we use the case 
$(\beta)$ 
of Lemma \ref{lemma2}, and conclude that that this situation contributes 
to the set of self-intersections of
$f(K)$]

$\langle 6\rangle$ 
If
$g(y_1, y_2, y_3, z_1)g(y_1, y_2, y_3, z_2)<0$ 
and
$g(z_1, z_2, y_2, y_3)$, $g(z_1, z_2, y_3, y_1)$,
and
$g(z_1, z_2, y_1, y_2)$
are nonzero, but not all of them have one and the same sign, then go to 
$\langle 8\rangle$; 
else go to 
$\langle 7\rangle$.
[At this step we use the case 
$(\gamma)$ 
of Lemma \ref{lemma2},
and conclude that this situation makes no contribution to the set of 
self-intersections of
$f(K)$]

$\langle 7\rangle$
Add the line ``The case of 
$(\delta, \Delta)$ 
requires additional study'' to the output file 
$\texttt{out1}$
and go to 
$\langle 8\rangle$.
[This step corresponds to situations not covered by Lemma \ref{lemma2}]

$\langle 8\rangle$ 
If 
$(\delta, \Delta)$ 
is the last item in the complete search, then go to 
$\langle 9\rangle$; 
else choose the next pair 
$(\delta, \Delta)$ 
and go to 
$\langle 1\rangle$.

$\langle 9\rangle$ 
Save the output files, 
$\texttt{out1}$ 
and 
$\texttt{out2}$, 
and quit.

Another description of our algorithm is available in \cite{AV24}.

\section{Studying the embeddedness of the Steffen polyhedron $S_0$ via symbolic calculations}\label{sec4} 

We have implemented the algorithm described in Section \ref{sec3} as 
a program in the software system Wolfram Mathematica.
The full text of this program is available in \cite{AV25}.

Recall that throughout this article, the Steffen polyhedron
$S_0$ 
is the polyhedron constructed in Section \ref{sec2} from the development
shown in Fig.~\ref{fig1}.
In particular, in the coordinate system constructed in Section \ref{sec2}, the vertex 
$v_9$
of
$S_0$
is located on the $z$-axis and has a negative $z$-coordinate.

The lists \texttt{s} and \texttt{t}  of 1- and 2-dimensional simplices of 
$K$, 
mentioned in Section \ref{sec3}, are compiled directly from the development of 
$S_0$,
shown in Fig.~\ref{fig1}.
In Section \ref{sec3} we explained in detail how we find the coordinates of 
all the vertices
$v_1,\dots,v_9$ 
of 
$S_0$
in rational numbers or in radicals.
Using these coordinates, we prepare the list \texttt{ss} of the edges of 
$f(K)\subset\mathbb{R}^3$
with the coordinates of their endpoints and the list \texttt{tt} of the faces of 
$f(K)\subset\mathbb{R}^3$
with the coordinates of their vertices.

Using the lists \texttt{s}, \texttt{t}, \texttt{ss}, and \texttt{tt} as
input data for our program and performing all calculations symbolically 
(i.e. without using floating point arithmetic), we get two empty output files
\texttt{out1} and \texttt{out2}.
This means that the program does not find 
in 
$S_0$
intersections and cases requiring additional study.
On this basis, we consider the following theorem to be proved:

\begin{thrm} 
The Steffen polyhedron 
$S_0$ 
is embedded $($i.e., $S_0$ has no self-intersections$)$. \hfill$\square$
\end{thrm}

Since we know the coordinates of all the vertices of
$S_0$ 
in rational numbers or in radicals, it is not difficult for us to 
find its volume in radicals.
To do this, we represent the volume of 
$S_0$ 
as the sum of the oriented volumes of 14 tetrahedra, each of which has the 
origin of the coordinate system as its apex and some face of 
$S_0$
as its base. 
(Of course, the orientation of each tetrahedron must be inherited from 
a fixed orientation of 
$S_0$.)
Symbolic computations in Wolfram Mathematica show that the volume of 
$S_0$ 
is equal to 
$187\sqrt{83}/(6\sqrt{2})$. 
It is interesting to note that the same number, $187\sqrt{83}/(6\sqrt{2})$, 
is equal to the volume of the tetrahedron whose vertices are the vertices
$v_1$, $v_2$, $v_3$, 
and
$v_4$  
of the Steffen polyhedron 
$S_0$, see Figs. 1 and 2.

\section{Numerical study of the flex $\{S_t\}_{t\in (-\varepsilon,\varepsilon)}$}\label{sec5} 

Recall that, in the last paragraph of Section \ref{sec2}, we
defined a flex 
$\{S_t \}_{t\in(-\varepsilon,\varepsilon)}$ 
of the Steffen polyhedron 
$S_0$
by specifying the position of each vertex
$v_j(t)$
of the polyhedron
$S_t$.
Namely, we put
$v_j(t)=v_j$
for all
$j=1,\dots,4$
and every $t$ close enough to zero, and define 
$v_9(t)$ 
as the point lying on the circle
$\gamma$,
defined in the property (b) in Section \ref{sec2}, and such that the oriented angle 
$\angle v_9(t)0v_9$
is equal to
$t$ 
radians.
Since lengths of all edges of the Steffen polyhedron are known to us, the above data 
is sufficient to compute the coordinates of the points 
$v_j(t)$ 
for every
$j=5,\dots,8$ 
and every
$t$ 
close enough to zero, from the coordinates of the points 
$v_j(t)$, $j=1,\dots,4$, 
and
$v_9(t)$.

In Section \ref{sec4} we used Lemma \ref{lemma2} in order to conclude that
$S_0$
is embedded, i.e., our arguments were based on some combinations of inequalities
involving the values of the polynomial
$g$.
By continuity, the same inequalities hold true for the values of
$g$
calculated for
$S_t$
for every
$t$
close enough to zero.
Thus, 
$S_t$
is embedded for all such 
$t$.
Naturally, we want to have a quantitative estimate for the maximum 
$t$
for which 
$S_t$
is embedded.
 
In Section \ref{sec5} we present the results of our study of the problem 
``what is the maximum value of
$\varepsilon>0$
such that 
$S_t$ 
has no self-intersections for all 
$t\in (-\varepsilon, \varepsilon)$?''
In other words, in this Section we want to understand how big the angle 
$\angle v_9(t)0v_9$
can be made so that
$S_t$
is still embedded.

Unfortunately, we can only answer this question using floating point calculations.

We put 
$t = \arcsin(9/40) \approx 0.226943 \approx 13.0029^{\circ}$ 
and apply our program in Wolfram Mathematica mentioned in Section \ref{sec4},
which implements our algorithm described in Section \ref{sec3}.
Using floating point calculations we see that $S_t$ 
has no self-intersections.

Similarly, for 
$t = \arcsin(19/80) \approx 0.239791 \approx 13.739^{\circ}$ 
numerical calculations show that 
$S_t$ 
has self-intersections. 
Moreover, our program informs us that self-intersections occur because the edge 
$\{v_2(t),v_3(t)\}$ 
intersects the face 
$\{v_7(t),v_8(t),v_9(t)\}$, 
and  
$\{v_7(t),v_8(t)\}$ 
intersects 
$\{v_1(t),v_2(t),v_3(t)\}$.

Therefore, we can assert that 
$v_9(t)$
can be moved around the circle 
$\gamma$
per angle up to
$13^{\circ}$
in both directions from the point
$v_9=v_9(0)$
so that  
$S_t$
is embedded.
In other words, the range of the displacements of 
$v_9(t)$
for which 
$S_t$
is embedded,
is at least
$9\sqrt{2621}/100\approx 4.60761$.
Note that this value is comparable to 5, i.e., to the length of the shortest edges of
$S_t$. 

\section{Concluding remarks}\label{sec6}

The theory of flexible polyhedra began with Bricard's article \cite{Br97}.
Its heyday started after Connelly's article \cite{Co77}.
The most famous result  of the theory of flexible polyhedra states 
that the volume of any flexible closed orientable polyhedron in
$\mathbb{R}^n$, $n\geq 3$,
remains constant during the flex, see,
for example, Gaifullin's overview article \cite{Ga18}.
This theory is also being developed in non-Euclidean spaces, 
see, for example, \cite{Ga25}.

The theory of flexible polyhedra is attractive because everyone can 
find here an open problem of any level of difficulty that would suit their
mathematical tastes and background.
Below we pose two open problems related to the topic of this article.

Recall that two embedded polyhedra in 
$\mathbb{R}^3$
are called scissors-congruent if the finite part of the space, bounded by the 
first polyhedron can be cut into finitely many tetrahedra that can be reassembled 
to yield the finite part of the space bounded by the second polyhedron.

The first problem develops the well-known fact that if one embedded polyhedron is 
obtained from another by a flex, then they are scissors-congruent, see \cite{GI18}.
Hence, the Steffen polyhedron 
$S_0$
is scissors-congruent to every polyhedron
$S_t$
constructed in the last paragraph of Section \ref{sec2} with 
$t$
close enough to zero (so that 
$S_t$
is embedded).

\textbf{Open problem 1:}
For a given 
$t$
close enough to zero, explicitly specify the partition of 
$S_0$
into a finite set of tetrahedra from which  
$S_t$
can be reassembled.

In Section \ref{sec4} we observed that the volume of the Steffen polyhedron 
$S_0$
is equal to the volume of the tetrahedron, whose vertices are the vertices
$v_1$, $v_2$, $v_3$, 
and
$v_4$  
of 
$S_0$.
This naturally gives rise to the following

\textbf{Open problem 2:}
Are the Steffen polyhedron 
$S_0$
and 
the tetrahedron 
$\{v_1, v_2, v_3, v_4\}$  
scissors-congruent?
If yes, explicitly specify the partition of 
$\{v_1, v_2, v_3, v_4\}$ 
into a finite set of tetrahedra from which  
$S_0$
can be reassembled.

In conclusion, the authors declare that all symbolic and numerical calculations 
performed in the preparation of this article were executed using the computer 
software system Wolfram Mathematica 12.1 \cite{Wo99}, license 3322--8225.

\subsection*{Funding}
The research was carried out within the State Task to the Sobolev Institute of 
Mathematics (Project FWNF--2022--0005).

\end{document}